\documentclass[a4paper, 11pt]{amsart}
\usepackage{amssymb,amsmath}

\usepackage[top=25truemm,bottom=25truemm,left=25truemm,right=25truemm]{geometry}
\usepackage{color}

\makeatletter
\@namedef{subjclassname@2010}{%
  \textup{2010} Mathematics Subject Classification}
\makeatother

\newcommand{\A}{\mathbb{A}}
\newcommand{\C}{\mathbb{C}}

\newcommand{\Z}{\mathbb{Z}}

\newcommand{\rank}{\operatorname{rank}}

\newcommand{\trdeg}{{\rm tr.deg}}
\newtheorem{thm}{Theorem}[section]
\newtheorem{prop}[thm]{Proposition}
\newtheorem{lem}[thm]{Lemma}
\newtheorem{cor}[thm]{Corollary}

\newtheorem{question}[thm]{Question}

\begin{document}
\title[Retracts of Laurent polynomial rings]{Retracts of Laurent polynomial rings}
\author{Neena Gupta and  Takanori Nagamine}
\address[Neena Gupta]
{Theoretical Statistics and Mathematics Unit, Indian Statistical Institute, 
203 B.T.Road, Kolkata-700108, India}
\email{neenag@isical.ac.in, rnanina@gmail.com}
\address[Takanori Nagamine]
{National Institute of Technology (KOSEN), Oyama College, 
771, Nakakuki, Oyama city, Tochigi 323-0806, Japan}
\email{t.nagamine14@oyama-ct.ac.jp}
\date{\today}
\thanks{The work of the second author was supported by JSPS KAKENHI Grant Number JP21K13782.}

%%%%%%%%%%%%%%%%%%%%%%%%%%%%%%%%%%%%%%%%%
%%%%%%%%%%%%%%%%%%%%%%%%%%%%%%%%%%%%%%%%%
%Abstract
%%%%%%%%%%%%%%%%%%%%%%%%%%%%%%%%%%%%%%%%%
%%%%%%%%%%%%%%%%%%%%%%%%%%%%%%%%%%%%%%%%%
\begin{abstract}

Let $R$ be an integral domain and $B=R[x_1,\ldots,x_n]$ be the polynomial ring. In this paper, we consider retracts of $B[1/M]$ for a monomial $M$. We show that (1) if $M=\prod_{i=1}^nx_i$, then every retract is a Laurent polynomial ring over $R$, (2) if $R$ is a UFD and $n\leq 3$, then every retract is isomorphic to $R[y_1^{\pm1},\ldots,y_s^{\pm1},z_1,\ldots,z_t]$ for some $s,t\geq 0$. 
\end{abstract}

\maketitle

\setcounter{section}{0}

%%%%%%%%%%%%%%%%%%%%%%%%%%%%%%%%%%%%%%%%%
%%%%%%%%%%%%%%%%%%%%%%%%%%%%%%%%%%%%%%%%%
%Section 1
%%%%%%%%%%%%%%%%%%%%%%%%%%%%%%%%%%%%%%%%%
%%%%%%%%%%%%%%%%%%%%%%%%%%%%%%%%%%%%%%%%%
\section{Introduction}

Throughout the paper, all rings are commutative with unity and any domain is understood to be an integral domain. 
Let $R$ be a domain and $A\subset B$ be $R$-algebras. $R^*$ is the group of units of $R$. $Q(R)$ is the quotient field of $R$. 
For an integer $n\geq 0$, $R^{[n]}$ (resp. $R^{[\pm n]}$) denotes the polynomial ring (resp. Laurent polynomial ring) in $n$ variables over $R$. For  $f\in R$, $R_f=S^{-1}R$ where $S=\{1,f,f^2,\ldots\}$. For a field $K$, $K^{(n)}$ denotes the field of fractions of the polynomial ring $K^{[n]}$. 
$A$ is called an {\bf $R$-retract} of $B$ if there is an $R$-algebra homomorphism $\varphi:B\to A$ such that $\varphi|_A={\rm id}_A$. 
In particular, $\varphi$ is called an {\bf $R$-retraction}. When $A$ is a domain, $\trdeg_R\:A$ denotes the transcendence degree of $Q(A)$ over $Q(R)$.

Let $k$ be a field and $A$ be a finitely generated $k$-domain of transcendence degree $d$ over $k$. We consider the following four conditions on $A$.  
\begin{enumerate} 
\item[{\bf (a)}] 
$A\cong_kk^{[d]}$, 
\item[{\bf (b')}]
$A^{[1]}\cong_kk^{[d+1]}$, 
\item[{\bf (b)}]
$A^{[m]}\cong_kk^{[d+m]}$ for some $m\geq 1$, 
\item[{\bf (c)}]
$A$ is a $k$-retract of $k^{[d+m]}$ for some $m\geq 1$. 
\end{enumerate}
It is clear that (a)$\implies$(b')$\implies$(b)$\implies$(c). 
Let $n=d+m$ for some $m\geq 1$. In \cite{Cos77}, Costa asked the following question: \emph{Is every retract of $k^{[n]}$ a polynomial ring over $k$?, i.e., does the implication {\rm(c)}$\implies${\rm(a)} hold?} He proved that the question is affirmative when $n\leq 2$ (\cite[Theorem 3.5]{Cos77}). The second author showed that this is affirmative when $n=3$ and the characteristic of $k$ is zero (\cite[Theorem 2.5]{Nag19}). Note that the first author and  Chakraborty, Dasgupta and Dutta showed the same result in \cite[Theorem 5.8]{CDDG21}. However, when $n=3$ and the characteristic of $k$ is positive, the question is still open.

The famous Zariski Cancellation Problem (ZCP) asks: \emph{Does the implication {\rm (b')}$\implies${\rm (a)} hold?} Costa's question is a generalization of ZCP.
We know that ZCP has affirmative answer when $d\leq 2$ (Abhyankar, Heinzer and Eakin \cite{AHE72}, Fujita \cite{Fuj79}, Miyanishi and Sugie \cite{MS80}, Russell \cite{Rus81}, the first author and Bhatwadekar \cite{BG15}, Kojima \cite{Koj16}). 
However, if $d\geq3$ and the characteristic of $k$ is positive, the first author gave counterexamples in \cite{Gup14A} and \cite{Gup14B}. In particular, these counterexamples show that when $d\geq3$, for any $m\geq 1$, retracts of $k^{[d+m]}$ need not be polynomial rings.

In this article, we consider similar problems for retracts of Laurent polynomial ring $k^{[\pm d]}$.
% another natural problems for retracts in which $k^{[d]}$ is replaced by the Laurent polynomial ring $k^{[\pm d]}$. 
This is an analogue of Costa's question and ZCP. More generally, we consider the following question. 

\begin{question}
{\rm 
Let $M\in k[x_1,\ldots,x_n]\cong_kk^{[n]}$ be a monomial. What are the retracts of $k[x_1,\ldots,x_n]_M$? 
}
\end{question}

Similarly, we consider the following three conditions on $A$.
\begin{enumerate} 
\item[{\bf (a*)}] 
$A\cong_kk^{[\pm d]}$, 
\item[{\bf (b*)}]
$A^{[\pm m]}\cong_kk^{[\pm(d+m)]}$ for some $m\geq 1$, 
\item[{\bf (c*)}]
$A$ is a $k$-retract of $k^{[\pm(d+m)]}$ for some $m\geq 1$. 
\end{enumerate}
It is clear that (a*)$\implies$(b*)$\implies$(c*) hold. The first author and Bhatwadekar \cite[Lemma 4.5]{BG12} proved that the implication {\rm (b*)}$\implies${\rm (a*)} holds, that is, the Zariski Cancellation Problem for a Laurent polynomial ring is affirmative. When $k=\C$, the same result was observed by Dubouloz in \cite{Dub16}. Freudenburg also discussed Laurent cancellation problems in \cite{Fre14}. 

In this paper, we show that the implication  (c*)$\implies$(a*) holds. More precisely, we show the following, which is the main result of this paper. 
%Main Theorem%
\begin{thm} \label{main}
Let $R$ be an integral domain. Then every $R$-retract of $R^{[\pm n]}$ is a Laurent polynomial ring over $R$ for any $n\geq 0$.  
\end{thm}
Therefore, conditions (a*), (b*) and (c*) are equivalent to each other. Moreover, in Corollary \ref{ret of lowvar} (c), we give a classification of the retracts of $R[x_1,x_2,x_3]_M\cong_R R^{[\pm d]}\otimes_R R^{[3-d]}$ for a UFD $R$ and $1\leq d\leq3$. Although Costa's original question is open for $k[x_1,x_2,x_3]$ in positive characteristic, we get a classification for $R[x_1,x_2,x_3]_M$ in any characteristic.  

Let $K/k$ be a field extension. $K$ is called {\bf rational} over $k$ if $K\cong_kk^{(n)}$ for some $n\geq0$. $K$ is called {\bf stably rational} over $k$ if $K^{(m)}$ is rational for some $m\geq0$. $K$ is called {\bf retract rational} over $k$ if there exist a $k$-domain $A$ such that $Q(A)=K$ and $A$ is a $k$-retract of $k[x_1,\ldots,x_n]_f$ for some $n\geq 0$ and $f\in k[x_1,\ldots,x_n]\cong_kk^{[n]}$. It is well known that if $k$ is an infinite field, then ``rational" $\implies$ ``stably rational" $\implies$ ``retract rational" (see e.g., \cite[Proposition 3.6 (a)]{Sal84}).  
The following question is an analogue of \cite[Question 4]{CDDG21}. 
\begin{question} \label{rationality}
{\rm
Let $M\in k[x_1,\ldots,x_n]\cong_kk^{[n]}$ be a monomial and $A$ be a $k$-retract of $k[x_1,\ldots,x_n]_M$. Does it follow that $Q(A)$ is rational over $k$?
}
\end{question}

Note that, if $A$ is a $k$-retract of $k[x_1,\ldots,x_n]_M$, then $Q(A)$ is retract rational over $k$. We give a partial answer of Question \ref{rationality} as below. 

\begin{thm} \label{rational}
Let $k$ be a field, $M\in k[x_1,\ldots,x_n]\cong_kk^{[n]}$ be a monomial and $A$ be a $k$-retract of $B:=k[x_1,\ldots,x_n]_M$. If one of the following holds, then $Q(A)$ is rational over $k$.
\begin{enumerate}
\item[{\bf (a)}]
$\trdeg_k\:A\in\{0,1,n\}$.
\item[{\bf (b)}]
$\rank(B^*/k^*)\geq n-2$. 
\item[{\bf (c)}]
$n\leq3$. 
\end{enumerate}
\end{thm}

This paper is organized as follows. In Section 2, we recall basic notions and results used in this paper. 
In Section 3, we study retracts of a localized polynomial ring over a domain. Then we give proofs of Theorems \ref{main} and \ref{rational}. 
%%%%%%%%%%%%%%%%%%%%%%%%%%%%%%%%%%%%%%%%%
%%%%%%%%%%%%%%%%%%%%%%%%%%%%%%%%%%%%%%%%%
%Section 2
%%%%%%%%%%%%%%%%%%%%%%%%%%%%%%%%%%%%%%%%%
%%%%%%%%%%%%%%%%%%%%%%%%%%%%%%%%%%%%%%%%%
\section{Preliminaries}

Let $R$ be a domain and $A\subset B$ be $R$-algebras. Note that the coefficient ring of the (Laurent) polynomial ring is a retract. In particular, the following holds.

\begin{lem}\label{ac}
If $B\cong_RR^{[\pm m]}\otimes_RR^{[n]}$ for some $m,n\geq 0$, then $R$ is an $R$-retract of $B$. 
\end{lem}
\begin{proof}
Omitted. 
\end{proof}

An element $b\in B$ is said to be {\bf algebraic over $A$} if $f(b)=0$ for some $0\not=f\in A^{[1]}$. $A$ is {\bf algebraically closed in $B$} if every algebraic element of $B$ over $A$ belongs to $A$. 

We recall a few important properties of retracts recorded by Costa in \cite{Cos77}.

%We recall some equivalent conditions for retracts. 

\begin{prop} {\rm (cf.}\:\cite[Proposition 1.1]{Cos77}{\rm )} \label{def}
Let $R$ be a domain and $A\subset B$ be $R$-algebras. The following conditions are equivalent. 
\begin{enumerate}
\item[{\bf (a)}]
$A$ is an $R$-retract of $B$. 
\item[{\bf (b)}]
There is an $R$-algebra homomorphism $\pi:B\to B$ such that $\pi^2=\pi$ and $\pi(B)=A$. 
\item[{\bf (c)}]
There is an ideal $I$ of $B$ such that $B=A\oplus I$ as an $A$-module. 
\end{enumerate}
\end{prop}

\begin{lem}\label{Costa}
{\rm (cf.}\:\cite[Lemma 1.3]{Cos77}{\rm )}
Let $B$ be a domain and  $A\subset B$ be an $R$-retract of $B$. Then $A$ is algebraically closed in $B$.
\end{lem}
By combining Lemmas \ref{ac} and \ref{Costa}, the coefficient ring $R$ of $B=R^{[\pm m]}\otimes_RR^{[n]}$ is algebraically closed in $B$. 
The following theorem characterizes retracts of polynomial rings in two variables over a UFD (\cite[Theorem 3.5 and subsequent Remark]{Cos77}).
\begin{thm} \label{/UFD}
{\rm (cf.}\:\cite[Theorem 3.5]{Cos77}{\rm )}
Let $R$ be a UFD and $A$ be a retract of $R^{[2]}$. Then $A\cong_RR^{[s]}$ for some $s\in\{0,1,2\}$.  
\end{thm}

%%%%%%%%%%%%%%%%%%%%%%%%%%%%%%%%%%%%%%%%%
%%%%%%%%%%%%%%%%%%%%%%%%%%%%%%%%%%%%%%%%%
%Section 3
%%%%%%%%%%%%%%%%%%%%%%%%%%%%%%%%%%%%%%%%%
%%%%%%%%%%%%%%%%%%%%%%%%%%%%%%%%%%%%%%%%%
\section{Proofs of Theorems \ref{main} and \ref{rational}}
Throughout this section, $R$ denotes a domain and $R[x_1,\ldots,x_n]\cong_RR^{[n]}$ the polynomial ring in $n\geq1$ variables over $R$. 
For an $R$-algebra $A$, define {\bf the multiplicative $\Z$-module $U_R(A)$} by $A^*/R^*$. A polynomial $M\in R[x_1,\ldots,x_n]$ is called a {\bf monomial} if $M\not\in R$ and $M$ is a product of variables. Note that, we may assume that, for some $1\leq d\leq n$, 
\[
R[x_1,\ldots,x_n]_M=R[x_1^{\pm1},\ldots,x_d^{\pm1},x_{d+1},\ldots,x_n].
\]
The following lemma gives fundamental properties of the multiplicative $\Z$-module $U_R(A)$. 
\begin{lem} \label{unit}
The following assertions hold true. 
\begin{enumerate}
\item[{\bf (a)}]
Let $\varphi:B\to A$ be an $R$-retraction. Then $\varphi$ induces the retraction $\tilde{\varphi}:U_R(B)\to U_R(A)$ as $\Z$-modules. Therefore, $U_R(A)$ is a direct summand of $U_R(B)$ as a $\Z$-module. 
\item[{\bf (b)}]
$U_R(R[x_1^{\pm1},\ldots,x_d^{\pm1},x_{d+1},\ldots,x_n])\cong\Z^{d}$. 
\end{enumerate}
\end{lem}
\begin{proof}
{\bf (a)} Since $\varphi$ is an $R$-algebra homomorphism, $\varphi|_{B^*}:B^*\to A^*$
induces the retraction $\tilde{\varphi}:U_R(B)\to U_R(A)$. Therefore, $U_R(B)=U_R(A) \oplus K$ for some $\Z$-module $K$. 

{\bf (b)} For $B=R[x_1^{\pm1},\ldots,x_d^{\pm1},x_{d+1},\ldots,x_n]$, $U_R(B)=\{x_1^{e_1}\cdots x_d^{e_d} \ | \ e_1,\ldots,e_d\in\Z \}$. Define the homomorphism $\alpha: U_R(B)\ni x_1^{e_1}\cdots x_d^{e_d}\to \sum_{i=1}^de_ix_i\in \bigoplus_{i=1}^d\Z x_i$. Then $\alpha$ is an isomorphism.  
\end{proof}

\subsection{Retracts of localized polynomial rings over a domain} \label{retract over a domain}
For $1\leq d\leq n$, let $B=R[{x_1}^{\pm 1}, \dots, {x_d}^{\pm 1}, x_{d+1}, \dots, x_n]$, $\varphi:B\to B$ be an $R$-algebra homomorphism such that $\varphi^2=\varphi$. Then $\rank(U_R(B))=d$, $A:=\varphi(B)$ is an $R$-retract of $B$ and $\varphi:B\to A$ is an $R$-retraction. Let $r=\rank(U_R(A))$. 

\begin{prop}\label{Gupta}
Assume the setup in Subsection \ref{retract over a domain}. The following assertions hold true.
\begin{enumerate}
\item[{\bf (a)}] 
If $\trdeg_R\:A=0$, then $A=R$. If $\trdeg_R\:A=n$, then $A=B$. 
\item[{\bf (b)}] There exist $y_1, \dots, y_d \in B$ such that 
	\begin{enumerate}
	\item[\rm(i)] $B= R[{y_1}^{\pm 1}, \dots, {y_d}^{\pm 1}, x_{d+1}, \dots, x_n]$ and
	\item[\rm(ii)] \begin{equation*}
\varphi(y_i)=
\left\{\begin{array}{lll}
y_i &{\text{~for~~}} 1\le i\le r\\
1 & {\text{~for~~}}  r+1\le i\le d
\end{array}
\right.
\end{equation*}
%$(\varphi(y_1),\ldots,\varphi(y_r),\varphi(y_{r+1}),\ldots,\varphi(y_d))=(y_1,\ldots,y_r,1,\ldots,1)$.
	\end{enumerate}
\item[{\bf (c)}] Let $S= R[{y_1}^{\pm 1}, \dots, {y_r}^{\pm 1}]$. Then $S \subseteq A$ and $A$ is an $S$-retract of\\ $B=S[{y_{r+1}}^{\pm 1}, \dots, {y_d}^{\pm 1}, x_{d+1}, \dots x_n]$ and $A=S[\varphi(x_{d+1}), \dots, \varphi{(x_n)}]$.
\item[{\bf (d)}] Let $J:=({y_{r+1}}-1, \dots, {y_d}-1)B$. Then $A$ is isomorphic to a subring $C$ of $B/J$, where $C$ is an $S$-retract of $B/J$. Therefore, identifying $A$ with $C$, and the fact that $B/J$ is isomorphic to $S^{[n-d]}$, we may regard $A$ as an $S$-retract of $S^{[n-d]}$. 
\end{enumerate}
\end{prop}
\begin{proof}
{\bf (a)} The assertions follow from Lemmas \ref{ac} and \ref{Costa}. 
%Suppose that $\trdeg_R\:A=0$. Then $A$ is algebraic over $R$, hence Lemma \ref{ac} implies that $A=R$. If $\trdeg_R\:A=n$, then $B$ is algebraic over $A$. It follows from Lemma \ref{Costa} (a) that $A=B$. 

{\bf (b)} By Lemma \ref{unit} (a), $U_R(A)$ is a direct summand of $U_R(B)$ and by Lemma \ref{unit} (b) $U_R(B)\cong{\Z}^d$. Therefore, $U_R(B){\Z}^d=U_R(A) \oplus K$ 
for some $\Z$-submodule $K$ of  $U_R(B)$. Here we identify $U_R(B)$ with $\Z^d$.  Since submodules of a free module over a PID are free, it follows that $U_R(A)$ and $K$ are free $\Z$-modules. Then $0\leq r\leq d$ and $\Z^d=U_R(B)=U_R(A) \oplus K= \Z^r\oplus \Z^{d-r}$. For $1\le i\le d$, let $b_{i}= (a_{i1}, \dots, a_{id})\in \Z^d$ be such that $\{b_i \ | \ 1\le i\le r\}$ is a $\Z$-basis of $U_R(A)$ and $\{b_i \ | \ r+1\le i\le d\}$ is a $\Z$-basis of $K$. Set $y_i:=x_1^{a_{i1}}\cdots x_d^{a_{id}}$ for $1\le i\le d$. Then the assertion (i) 
 follows with these $y_i$'s. 
Now for $1\le i\le r$, $\varphi(y_i)=y_i$ and hence $y_i \in A$. Suppose $r+1\le i\le d$.  Then $\varphi(y_i)=\lambda_i$ for some $\lambda_i \in R^*$. Replacing $y_i$ by $\lambda_i^{-1}y_i$, it follows that $\varphi(y_i)=1$.

{\bf (c)} The assertions follow from (b).  

{\bf (d)} Let $\pi: B \to B/J$ be the quotient map.  We note that $\varphi(J)=0$. Also by (b), $\varphi|_S$ is the identity ring homomorphism. Therefore, $\varphi$ induces a unique $S$-algebra homomorphism $\bar{\varphi}: B/J \to B/J$ such that $\bar{\varphi}\circ\pi=\pi\circ\varphi$. Since $\varphi^2=\varphi$, we have 
\[
(\bar{\varphi})^2\circ\pi=\bar{\varphi}\circ(\bar{\varphi}\circ\pi)=\bar{\varphi}\circ(\pi\circ{\varphi})=\pi\circ\varphi^2=\pi\circ\varphi.
\] 
Therefore, by uniqueness of $\bar{\varphi}$, we have $(\bar{\varphi})^2=\bar{\varphi}$. Thus, $C:= \bar{\varphi}(B/J)$ is an $S$-retract of $B/J$. Since $\varphi(B)=A$, we see that $\pi(A)=C$. We now show that $C$ is isomorphic to $A$. For this, it is enough to show that $\pi|_A$ is injective. 
Since $\varphi:B\to A$ is an $R$-retraction, $B=A\oplus\ker\varphi$ and hence $A\cap\ker\varphi=\{0\}$. Then $\ker(\pi|_A)=A\cap J\subset A\cap\ker\varphi=\{0\}$. 
Thus, $\pi|_A$ is injective.
\end{proof}

\begin{cor} \label{inequality}
Assume the setup in Subsection \ref{retract over a domain}. The following assertions hold true.
\begin{enumerate}
\item[{\bf (a)}]
$0\leq \trdeg_R\:A-r\leq n-d$. 
\item[{\bf (b)}]
If $\trdeg_R\:A=r$, then $A\cong_RR^{[\pm r]}$. 
\item[{\bf (c)}]
If $\trdeg_R\:A=r+n-d$, then $A\cong_RR^{[\pm r]}\otimes_RR^{[n-d]}$. 
\end{enumerate} 
\end{cor}
\begin{proof}
{\bf (a)} By Proposition \ref{Gupta} (c), $S=R[{y_1}^{\pm 1}, \dots, {y_r}^{\pm 1}]\subset A$ and $S\cong_RR^{[\pm r]}$, hence $\trdeg_R\:A\geq r$. Moreover, $A=S[\varphi(x_{d+1}), \dots, \varphi{(x_n)}]$ implies that $\trdeg_R\:A\leq r+n-d$. 

{\bf (b)} If $\trdeg_R\:A=r$, then $A$ is algebraic over $S$. By Proposition \ref{Gupta} (c), we may regard $A$ as an $S$-subalgebra retract of $B/J\cong_S S[x_{d+1}, \dots x_n]$. Since $S$ is algebraically closed in $S[x_{d+1}, \dots x_n]$, we have $A=S\cong_RR^{[\pm r]}$. 

{\bf (c)} Suppose that $\trdeg_R\:A=r+n-d$. Then $y_1,\ldots,y_r,\varphi(x_{d+1}),\ldots,\varphi(x_n)$ are algebraically independent over $R$. For any $d+1\leq j\leq n$, $\varphi(x_{j})\notin A^*$. Therefore $A\cong_RR^{[\pm r]}\otimes_RR^{[n-d]}$.  
\end{proof}

Now Theorem \ref{main} follows imeediately from Corollary \ref{inequality}.
\begin{proof}[Proof of Theorem \ref{main}]
Let $A$ be an $R$-retract of $B:=R[x_1^{\pm 1},\ldots,x_n^{\pm 1}]$ and $r=\rank (U_R(A))$. Since $\rank (U_R(B))=n$, Corollary \ref{inequality} (a) implies that $\trdeg_R\:A-r\leq n-n=0$ and hence $\trdeg_R\:A=r$. By Corollary \ref{inequality} (b), $A\cong_RR^{[\pm r]}$. 
\end{proof}

The following corollary shows that the Zariski Cancellation Problem for a Laurent polynomial ring is affirmative. Note that this was proved by the first author and Bhatwadekar in \cite[Lemma 4.5]{BG12}. 

\begin{cor} \label{st}
Let $A$ be an $R$-domain of the transcendence degree $d$ over $R$. If $A^{[\pm m]}\cong_R R^{[\pm (d+m)]}$ for some $m\geq1$, then $A\cong_R R^{[\pm d]}$. 
\end{cor} 
\begin{proof}
Since $A$ is an $R$-retract of $R^{[\pm (d+m)]}$, the assertion follows from Theorem \ref{main}. 
\end{proof}

The following corollary gives classifications of retracts of $R[x_1^{\pm1},\ldots,x_d^{\pm1},x_{d+1},\ldots,x_n]$ for some $1\leq d\leq n$ in special cases. These are analogues of Theorem \ref{/UFD}, \cite[Theorem 3.4]{Cos77}, \cite[Theorem 5.8]{CDDG21} and \cite[Theorem 2.5]{Nag19}. 

\begin{cor}\label{ret of lowvar}
Assume the setup in Subsection \ref{retract over a domain}. The following assertions hold true. 
\begin{enumerate}
\item[{\bf (a)}]
If $d\geq n-1$, then $A\cong_RR^{[\pm r]}\otimes_RR^{[s]}$ for $s\in\{0,1\}$.  
\item[{\bf (b)}]
Assume that $d=n-2$. If $R$ is a UFD, then $A\cong_RR^{[\pm r]}\otimes_RR^{[s]}$ for $s\in\{0,1,2\}$.  
\item[{\bf (c)}]
If either $n\leq 2$, or $n=3$ and $R$ is a UFD, then $A\cong_RR^{[\pm r]}\otimes_RR^{[s]}$ for $s\in\{0,1, 2\}$.  
\end{enumerate}
\end{cor}

\begin{proof}
{\bf (a)} 
Suppose that $d\geq n-1$. By Corollary \ref{inequality} (a), $0\leq \trdeg_R\:A-r\leq n-d\leq 1$ and hence $\trdeg_R\:A\in\{r,r+n-d\}$. Therefore, the assertion follows from Corollary \ref{inequality} (b) and (c).  

{\bf (b)} Since $n-d=2$, Proposition \ref{Gupta} (d) implies that $A$ is an $S$-retract of $S^{[2]}$, where $S\cong_RR^{[\pm r]}$. Since $S$ is a UFD, it follows from Theorem \ref{/UFD} that $A\cong_SS^{[s]}\cong_RR^{[\pm r]}\otimes_RR^{[s]}$ for $s\in\{0,1,2\}$.  

{\bf (c)} If $n\leq 2$, then $d\geq 1\geq n-1$. Therefore the assertion follows from (a). 
Suppose that $n=3$ and $R$ is a UFD. Since $d\geq 1=n-2$, the assertion follows from (a) and (b). 
\end{proof}

As a consequence of Corollary \ref{ret of lowvar}, Theorem \ref{rational} holds as below. 
\begin{proof}[Proof of Theorem \ref{rational}]
Suppose that (a) holds. If $\trdeg_k\:A\in\{0,n\}$, then it follows from Proposition \ref{Gupta} (a) that $Q(A)$ is rational. If $\trdeg_k\:A=1$, then it follows from L\"{u}roth's theorem that $Q(A)$ is rational. If (b) or (c) holds, then Corollary \ref{ret of lowvar} implies that $Q(A)$ is rational. 
\end{proof}

\medskip
\noindent {\bf Acknowledgments.} 
The work of the second author was supported by JSPS KAKENHI Grant Number JP21K13782.

%%%%%%%%%%%%%%%%%%%%%%%%%%%%%%%%%%%%%%%%%
%%%%%%%%%%%%%%%%%%%%%%%%%%%%%%%%%%%%%%%%%
%References
%%%%%%%%%%%%%%%%%%%%%%%%%%%%%%%%%%%%%%%%%
%%%%%%%%%%%%%%%%%%%%%%%%%%%%%%%%%%%%%%%%%

\end{document}